\documentclass[12 pt, reqno]{amsart}
\usepackage{mathrsfs}
\usepackage{graphicx}
\usepackage{amsmath}
\usepackage{amsfonts}
\usepackage{amssymb}
\usepackage{hyperref} 
\newtheorem{Exam.}{Exam.}

\newtheorem{theorem}{Theorem}[section]
\newtheorem{lemma}{Lemma}[section]
\newtheorem{proposition}{Proposition}[section]
\newtheorem{corollary}{Corollary}[section]
\newtheorem{Conjecture}{Conjecture}[section]

\theoremstyle{definition}
\newtheorem{remark}{Remark}

\usepackage{amsthm,amsmath,amssymb,url,cite,color}

\numberwithin{equation}{section}

\def\det{\operatorname{det}}
\def\N{\mathbb N}
\def\R{\mathbb R}

\def\x{{\bf x}}

\DeclareMathOperator{\JS}{JS}
\DeclareMathOperator{\js}{Jc}

\begin{document}

\title[Positivity properties of two combinatorial sequences]{Positivity properties of Jacobi-Stirling numbers \\
and  generalized Ramanujan polynomials}

\author{Zhicong Lin}
\address[Zhicong Lin]{Department of Mathematics and Statistics, Lanzhou University, China, and  Universit\'{e} de Lyon; Universit\'{e} Lyon 1; Institut Camille Jordan; UMR 5208 du CNRS; 43, boulevard du 11 novembre 1918, F-69622 Villeurbanne Cedex, France}
\email{lin@math.univ-lyon1.fr}

\author{Jiang Zeng}
\address[Jiang Zeng]{Universit\'{e} de Lyon; Universit\'{e} Lyon 1; Institut Camille Jordan; UMR 5208 du CNRS; 43, boulevard du 11 novembre 1918, F-69622 Villeurbanne Cedex, France}
\email{zeng@math.univ-lyon1.fr}

\date{\today}
\begin{abstract} 
Generalizing recent results of  Egge and  Mongelli, we   show that 
    each  diagonal  sequence of the Jacobi-Stirling numbers  $\js(n,k;z)$ and $\JS(n,k;z)$ is a P\'olya frequency sequence   if and only if  $z\in [-1, 1]$ and 
    study  the $z$-total positivity properties of 
 these numbers. 
    Moreover, the polynomial  sequences 
    $$\biggl\{\sum_{k=0}^n\JS(n,k;z)y^k\biggr\}_{n\geq 0}\quad \text{and} \quad 
   \biggl\{\sum_{k=0}^n\js(n,k;z)y^k\biggr\}_{n\geq 0}$$
     are proved to be strongly $\{z,y\}$-log-convex.
  In the same vein, 
  we extend  a recent result of Chen et al.  about the  Ramanujan polynomials to 
Chapoton's 
generalized Ramanujan polynomials.
 Finally,  bridging   the  Ramanujan polynomials and 
a sequence arising from the Lambert $W$ function, we obtain
 a neat proof of  
 the unimodality of the latter sequence, which was proved  previously by Kalugin and Jeffrey.
\end{abstract}


\maketitle


\section{Introduction}
The {\em Jacobi-Stirling numbers}  of the first kind  $\js(n,k;z)$  and of  the second kind $\JS(n,k;z)$ ($n\geq k\geq0$) are defined by
the recurrence relations:
\begin{align}
&\js(n,k;z)=\js(n-1,k-1;z)+(n-1)(n-1+z)\js(n-1,k;z)\label{eq:ja0},\\
&\JS(n,k;z)=\JS(n-1,k-1;z)+k(k+z)\JS(n-1,k;z)\label{eq:ja},
\end{align}
with the  boundary conditions $\JS(0,0;z)=\js(0,0;z)=1$ and $\JS(j,0;z)=\JS(0,j;z)=\js(j,0;z)=\js(0,j;z)=0$ for $j\geq1$.
The first values of these two sequences are given in Tables~\ref{Lnkz} 
and~\ref{jsnkz}.
When $z=1$, the two kinds of Jacobi-Stirling numbers are called 
the (unsigned) \emph{Legendre-Stirling numbers} of
the first and second kinds \cite{al,AGL11}. 

 Recently,  these numbers 
have attracted the attention of several authors~\cite{al,AGL11,gz,aegl,eg,GLZ,mo,mon2}.
In particular, a result of Egge~\cite[Theorem~5.1]{eg}  implies 
  that the diagonal sequences 
  $$
  \{\JS(k+n,n;1)\}_{n\geq 0}\quad\textrm{and} \quad
  \{\js(k+n,n;1)\}_{n\geq k}
  $$
   are P\'olya frequency sequences for any fixed $k\in \N$, while 
   Mongelli~\cite{mo} 
 studied total positivity properties of Jacobi-Stirling numbers  assuming that 
  $z$ is a real number. 
  
  \begin{table}
\caption{The first values of $\JS(n,k;z)$}
\label{Lnkz}
\begin{tiny} \[ \begin{tabular}{c|cccccc}
$k\backslash n$ & $1$ & $2$ & $3$ & $4$ & $5$  & $6$
\\
\hline
$1$ & $1$&$z+1$&$(z+1)^2$& $(z+1)^3$  &$(z+1)^4$&$(z+1)^5$ \\
$2$ && $1$& $5+3z$&$21+24z+7z^2$  &$85+141z+79z^2+15z^3$&$341+738z+604z^2+222z^3+31z^4$   \\
$3$ & &&$1$ &$14+6z$&$147+120z+25z^2$ &$1408+1662z+664z^2+90z^3$   \\
$4$ & &&&$1$ &  $30+10z$ &$627+400z+65z^2$                \\
$5$ & &&&&$1$  &$55+15z$        \\
$6$ & & &&&&$1$ \\
 \end{tabular} \] \end{tiny} \end{table}

  It is convenient to recall some necessary definitions.
A sequence of nonnegative real numbers 
$\{a_n\}_{n\geq0}$
 is  \emph{unimodal} if $a_0\leq\cdots\leq a_{m-1}\leq a_m\geq a_{m+1}\geq\cdots$ for 
 some $m$, and is \emph{log-concave} (resp.~\emph{log-convex}) if $a_i^2\geq a_{i-1}a_{i+1}$ (resp.~$a_i^2\leq a_{i-1}a_{i+1}$) for all $i\geq 1$.  
 A real sequence $\{a_n\}_{n \geq 0}$ is called a \emph{P\'olya frequency sequence} (PF sequence for short) if the matrix 
  $M:=(a_{j-i})_{i,j\geq 0}$ (where $a_k=0$ if $k<0$)  is \emph{totally positive} 
  (TP for short), that is,  every minor of $M$ is nonnegative. Unimodal, log-concave and P\'olya frequency sequences arise often in combinatorics~\cite{Brenti89}.

The following is our result about diagonal sequences of Jacobi-Stirling numbers.
 \begin{theorem}\label{PF:diagonal}
For any  fixed integer $k\geq1$,  the two sequences $\{\JS(k+n,n;z)\}_{n\geq0}$ and $\{\js(k+n,n;z)\}_{n\geq 0}$ are P\'olya frequency sequences  if and only if $-1\leq z\leq1$.
\end{theorem}

For a sequence of polynomials in ${\mathbf x}=\{x_1,x_2,\ldots, x_n\}$, 
  one  can define the ${\bf x}$-analog of log-concavity, log-convexity, total positivity and P\'olya frequency sequence as follows (see~\cite{sa,LW10,CWY11}).
Let $\R_{+}=\{x\in\R : x\geq0\}$.
   Given two polynomials $f({\bf x}), g({\bf x})\in \R_+[{\bf x}]$, we define 
$$
f({\bf x})\leq_{{\bf x}}g({\bf x})\quad \text{if and only if}\quad g({\bf x})-f({\bf x})\in \R_{+}[{\bf x}].
$$
 A sequence of polynomials $\{f_k({\bf x})\}_{k\geq0}$ in $\R_+[{\bf x}]$  is called \emph{${\bf x}$-log-concave} if 
$$
f_{k-1}({\bf x})f_{k+1}({\bf x})\leq_{{\bf x}}f_k({\bf x})^2\text{ \, for all $k\geq1$},
$$
and it is \emph{strongly ${\bf x}$-log-concave} if 
$$
f_{k-1}({\bf x})f_{l+1}({\bf x})\leq_{{\bf x}}f_{k}({\bf x})f_{l}({\bf x})\text{ \, for all $l\geq k\geq1$}.
$$
The {\em ${\bf x}$-log-convexity} and {\em strong ${\bf x}$-log-convexity} are defined similarly. 

\begin{remark}
For a sequence of real numbers 
$\{a_n\}_{n\geq0}$, the log-concavity  is equivalent to the strong log-concavity, that is, $a_{k-1}a_{l+1}\leq a_ka_l$ for all $l\geq k\geq1$. But,
for  polynomial sequences, 
  the ${\bf x}$-log-concavity is not equivalent to strong ${\bf x}$-log-concavity(see \cite{sa0}), 
  which is the same for ${\bf x}$-log-convexity and strong ${\bf x}$-log-convexity (see \cite{CWY11}).
\end{remark}

 A matrix $F=(f_{i,j})_{i,j\in\N}$, where $f_{ij}\in \R_+[{\bf x}]$,  is called \emph{${\bf x}$-totally positive} if every minor of $F$ is nonnegative with respect to $\geq_{{\bf x}}$. The \emph{${\bf x}$-P\'olya frequency sequence} is defined similarly. Note that if a sequence $\{f_k({\bf x})\}_{k\geq0}$ is a 
 ${\bf x}$-PF  sequence, then it is strongly ${\bf x}$-log-concave, that is,
$$
\left|\begin{array}{cc}
f_{k}(\x) & f_{l+1}(\x)\\
f_{k-1}(\x) & f_{l}(\x)
\end{array}
\right|\geq_{\x}0.
$$
In particular, we say that 
the finite  sequence $f_0, f_1, \ldots, f_d$
 is unimodal (respectively, log-concave,~etc.),
if the corresponding sequence $\{f_n\}_{n\geq 0}$,  with $f_n=0$ for $n>d$ enjoys the corresponding property.

  \begin{table}
\caption{The first values of $\js(n,k;z)$}
\label{jsnkz}
\begin{tiny} \[ \begin{tabular}{c|ccccc}
$k\backslash n$ & $1$ & $2$ & $3$ & $4$ & $5$
\\
\hline
$1$ & $1$       & $z+1$   & $2z^2+6z+4$    & $6z^3+36z^2+66z+36$ & $24z^4+240z^3+840z^2+1200z+576$
\\ $2$ &  & $1$ & $3z+5$  & $11z^2+48z+49$ & $50z^3+404z^2+1030z+820$
\\ $3$ & & &$1$ & $6z+14$ & $35z^2+200z+273$
 \\$4$ & &&&$1$ & $10z+30$
 \\$5$ & &&&&$1$
\\
 \end{tabular} \] \end{tiny} \end{table}

In this paper we will  prove the following results  about the ${\mathbf x}$-positivity properties of the Jacobi-Stirling numbers.
\begin{theorem}\label{th:ja}
For rows and columns of Jacobi-Stirling numbers, we have
\begin{itemize}
\item[(i)] Fix $n\in \N$, the sequence $\{\JS(n,k;z-1)\}_{k=0}^n$ is strongly $z$-log-concave.
\item[(ii)] Fix  $k\in\N$, the sequence  $\{\JS(n,k;z-1)\}_{n\geq k}$ is a $z$-PF sequence. 
\item[(iii)] Fix  $n\in\N$, the sequence $\{\js(n,k;z-1)\}_{k=1}^{n}$ is a $z$-PF sequence. 
\end{itemize}
\end{theorem}

It follows from (ii)  and (iii) of Theorem~\ref{th:ja}  that  the sequences $\{\JS(n,k;z-1)\}_{n\geq k}$ and $\{\js(n,k;z-1)\}_{k=1}^{n}$
are strongly $z$-log-concave. As pointed out in~\cite{mo}, the sequence
$\{\js(n,k;z-1)\}_{n\geq k}$ is even not log-concave for real value $z$.

\begin{theorem}\label{total:jc}
The three matrices $(\JS(n,k;z-1))_{n,k\geq0}$,  
 $(\js(n,n-k;z-1))_{n,k\geq0}$ 
 and  $(\js(n,k;z-1))_{n,k\geq0}$ are  $z$-totally positive.
\end{theorem}

\begin{theorem}\label{bell:JS}  For column generating functions of Jacobi-Stirling numbers, we have
\begin{itemize}
\item[(i)] The polynomial sequence $\left\{\sum_{k=0}^n\JS(n,k;z)y^k\right\}_{n\geq0}$ is 
strongly $\{z,y\}$-log-convex.
\item[(ii)] The polynomial sequence $\left\{\sum_{k=0}^n\js(n,k;z)y^k\right\}_{n\geq0}$ is strongly $\{z,y\}$-log-convex.
\end{itemize}
\end{theorem}

In this paper we shall also study the ${\mathbf x}$-positivity properties of  a polynomial sequence related to Ramanujan and Lambert.
It is well known that Lambert's equation $we^{-w}=y$ has an explicit solution  $w=\sum_{n\geq 1} n^{n-1}y^n/n!$. 
Note that  the coefficient $n^{n-1}$ is the number of rooted trees on $n$ vertices.
It is also known (see \cite{Zeng99,GZ})  that  the $n$-th derivation (with respect to $y$) of Lambert's function has the following formula
\[
w^{(n)}=\frac{e^{nw}}{(1-w)^n}R_n\left(\frac{1}{1-w}\right),
\]
where $R_n(y)$ are  the so-called Ramanujan polynomials  defined by the recurrence relation 
\begin{align}\label{def:R}
R_1(y)=1,\quad R_{n+1}(y)=n(1+y)R_n(y)+y^2R'_n(y).
\end{align}
The first values of the polynomials $R_n$ are
$$
R_2(y)=1+y,\quad R_3(y)=2+4y+3y^2,\quad R_4(y)=6+18y+25y^2+15y^3.
$$
It is clear that $R_n(y)$ is a polynomial in $y$ of degree $n-1$ with positive integral coefficients such that 
$R_n(0)=(n-1)!$, $R_n(1)=n^{n-1}$ and the coefficient of $y^{n-1}$ is $(2n-3)!!$.  Actually all the coefficients of $R_n(y)$ have nice combinatorial interpretation on trees\cite{Zeng99}. As we will show, the Ramanujan polynomials can be used to give a new  proof  of  a recent unimodal result of   Kalugin and Jeffrey~\cite{KJ10}.
Chapoton (see\cite{GZ})  introduced the  \emph{generalized Ramanujan polynomials} $Q_n(x,y,z,t)$  defined by
\begin{align}\label{def:Q}
Q_1=1,\quad Q_{n+1}=\left[x+nz+(y+t)(n+y\partial_y\right)]Q_n.
\end{align}
For example, we have $Q_2(x,y,z,t)=x+y+z+t,$ and 
\begin{align*}
Q_3(x,y,z,t)=x^2+3xy+3xz+3xt&+3y^2+4yz+5yt+2z^2+4zt+2t^2.
\end{align*}
Clearly,  comparing \eqref{def:R} with \eqref{def:Q} we have
\begin{equation}\label{spec:ramanu:}
R_n(y)=Q_n(0,y,1,0).
\end{equation}
Combinatorial interpretations of $Q_n$ in terms of plane trees and forests are given  in~\cite{GZ} as well as some other remarkable properties. 
Motivated by the recent result of 
Chen et al.~\cite{CWY11}  about
 $\{R_n(y)\}_{n\geq1}$, we 
shall prove the ${\bf x}$-log-convexity of  the polynomials $Q_n$.
\begin{theorem} \label{convex}
The sequence  $\{Q_n(x,y,z,t)\}_{n\geq1}$ is strongly ${\bf x}$-log-convex, that is, for any $n\geq m\geq 2$,
$$Q_{m-1}(x,y,z,t)Q_{n+1}(x,y,z,t)-Q_m(x,y,z,t)Q_n(x,y,z,t)\in\N[x,y,z,t].$$
\end{theorem}
\begin{remark}
Setting $x=0,z=1$ and $t=0$ we recover Chen et al.'s result about strong $y$-log-convexity of $R_n(y)$~\cite{CWY11}.
\end{remark}

This paper is organized as follows. In section~\ref{pf:jacobi}, we study the PF property of diagonal Jacobi-Stirling numbers and give a proof of Theorem~\ref{PF:diagonal}
with the parameter $z$ being a real number. In section~\ref{total:jacobi}, we investigate the $z$-total positivity of Jacobi-Stirling numbers.  In section~\ref{stro:log-conv}, we study the strong $\x$-log-convexity of the generating functions of Jacobi-Stirling numbers and generalized Ramanujan polynomials. In section~\ref{lambert}, we show that the unimodality of a sequence arising from  Lambert $W$ function first proved by Kalugin and Jeffrey~\cite{KJ10} follows easily from the log-concavity of the coefficients of Ramanujan polynomials.

\section{PF properties of diagonal Jacobi-Stirling numbers} \label{pf:jacobi}
Our main tool is the following result, due to Brenti\cite[Theorem 4.5.3]{Brenti89}, characterizing the rational formal power series whose coefficients are PF sequence.
\begin{lemma}\label{rfpo}
Let  $\sum_{n\geq 0}a_nx^n=P(x)/Q(x)$, where $P(x)$ and $Q(x)$ are  two relatively prime polynomials.
Then $\{a_n\}_{n\geq0}$ is a PF sequence if and only if 
\begin{enumerate}
\item  $a_n\geq0$ for all $n\geq0$,
\item $P(x)$ has only real nonpositive zeros, 
\item $Q(x)$ has only real positive zeros.
\end{enumerate}
\end{lemma}

We start with some  preliminary results about the generating function of the diagonal sequence of the Jacobi-Stirling numbers:
$$
F_k(x;z)=\sum_{n\geq0}\JS(k+n,n;z)x^n, \qquad k\geq 0.
$$
\begin{lemma}\label{rational:22}
For any fixed $z\in \R\setminus\{1\}$ and $k\geq 0$, there exists a polynomial $A_k(x;z)$ in $x$ of degree $2k$ such that
\begin{equation}\label{rational}
F_k(x;z)=\frac{A_k(x;z)}{(1-x)^{3k+1}}
\end{equation}
and $A_k(1;z)\neq0$.
\end{lemma}
\begin{proof}
 For $n\geq 0$, let $f_k(n;z)=\JS(k+n,n;z)$.
Then recurrence~\eqref{eq:ja} can be written as
\begin{equation}\label{pJS}
f_k(n;z)-f_k(n-1;z)=n(n+z)f_{k-1}(n;z)\qquad (k\geq 0)
\end{equation}
with $f_0(n;z)=1$ and $f_{-1}(n;z)=0$. We prove by induction on $k$ that $f_k(n;z)$ is a polynomial in $n$ of degree $3k$ if $z\neq1$. 
This is clear for $k=0$. Suppose $k\geq1$. By induction hypothesis 
 the right-hand side of \eqref{pJS}
 is a polynomial in $n$ of degree $3k-1$.
  Since the left-hand side of     \eqref{pJS}  is the difference of $f_k(n;z)$,
then $f_k(n;z)$ is a polynomial in $n$ of degree $3k$.  By a standard result about  the  generating functions of polynomial sequences (cf.~\cite[Corollary 4.3.1]{st1})
there exists a polynomial $A_k(x;z)$ in $x$ of degree  $\leq 3k$ satisfying~\eqref{rational} and $A_k(1;z)\neq0$.
By~\cite[Proposition~4.2.3]{st1}, we have 
$$
\sum_{n\geq 1}f_k(-n;z)x^n=-F_k(1/x;z)=- \frac{x^{3k+1}A_k(1/x;z)}{(x-1)^{3k+1}}.
$$
For $k\geq1$ it is clear that the degree of $A_k(x;z)$ must be $2k$ provided  that
\begin{equation}\label{condi}
f_k(0;z)=f_k(-1;z)=\ldots=f_k(-k;z)=0\ \ \text{and}\ \ f_k(-k-1;z)\neq0.
\end{equation}
We verify \eqref{condi} by induction on $k\geq 1$. First, 
from~\eqref{pJS}  we derive that 
$$f_1(n;z)=\frac{n(n+1)}{2}\left(\frac{2n+1}{3}+z\right).
$$ 
Hence  $f_1(0,z)=f_1(-1,z)=0$ and $f_1(-2,z)=z-1\neq 0$.
Assume  that $k\geq2$ and \eqref{condi} holds for $k-1$, i.e.,
$f_{k-1}(n;z)=0$ for $0\geq n\geq -k+1$ and $f_{k-1}(-k;z)\neq0$. 
By definition $f_k(0;z)=\JS(k,0;z)=0$,  hence
we can derive \eqref{condi} from \eqref{pJS} and  the induction hypothesis.
\end{proof}
By Lemma~\ref{rational:22} we can write $A_k(x;z)$ in~\eqref{rational} as
 \begin{align}\label{eq:A}
 A_k(x;z)=\sum_{i=1}^{2k}a_{k,i}(z)x^i.
 \end{align} 
\begin{proposition}
The coefficients  $a_{k,i}(z)$ in \eqref{eq:A} satisfy the following recurrence
\begin{align}
a_{k,i}(z)=i(i+z)a_{k-1,i}(z)&+[2i(3k-i-1)-(1-z)(3k-2i)]a_{k-1,i-1}(z)\nonumber\\
&+(3k-i)(3k-i-z)a_{k-1,i-2}(z),\label{recurrence}
\end{align}
with $a_{0,i}(z)=\delta_{0,i}$.
Thus,  when $-1<z<1$,  the coefficients $a_{k,i}(z)$ are  nonnegative for $k\geq1$ and $1\leq i\leq 2k$.
\end{proposition}
\begin{proof}
 For $k\geq1$, by \eqref{pJS}, we have
\begin{align}\label{eq:pf}
F_k(x;z)&=\sum_{n\geq1}(f_k(n-1;z)+n(n+z)f_{k-1}(n;z))x^n\nonumber\\ 
&=xF_k(x;z)+xD(x^{1-z}D(x^zF_{k-1}(x;z)))\nonumber\\
&=\frac{x}{1-x}D(x^{1-z}D(x^zF_{k-1}(x;z))),
\end{align}
where $D=\frac{d}{dx}$ and $F_0(x;z)=(1-x)^{-1}$. 
Substituting \eqref{rational} into \eqref{eq:pf} we obtain
\begin{align*}
(1-x)^{-3k-1}\sum_{i=1}^{2k}a_{k,i}(z)x^i=x(1-x)^{-1}D[x^{1-z}D[(1-x)^{-3k+2}\sum_{i=1}^{2k-2}a_{k-1,i}(z)x^{i+z}]],
\end{align*}
which is simplified to 
\begin{align*}
\sum_{i=1}^{2k}a_{k,i}(z)x^i=&(3k-1)(3k-2)\sum_{i=1}^{2k-2}a_{k-1,i}(z)x^{i+2}+(1-x)^2\sum_{i=1}^{2k-2}i(i+z)a_{k-1,i}(z)x^{i}\\
&+(3k-2)(1-x)\sum_{i=1}^{2k-2}(2i+1+z)a_{k-1,i}(z)x^{i+1}.
\end{align*}
Taking the coefficient of $x^i$ in both sides of the above equation, we get~\eqref{recurrence}.

For $1\leq i\leq 2k$ and $-1<z<1$ it is easy to verify that
$$
i(i+z)\geq0,\quad 
(3k-i)(3k-i-z)\geq0, \quad 2i(3k-i-1)-(1-z)(3k-2i)\geq0.
$$
Hence,  by \eqref{recurrence}, the coefficients $a_{k,i}(z)$ are nonnegative for $1\leq i\leq 2k$. This finishes the proof of the lemma.
\end{proof}
\begin{lemma}\label{PF:JSD}
For $-1\leq z\leq1$,  the zeros of the polynomial $A_k(x;z)$ in~\eqref{rational} are  distinct, real and nonpositive numbers.
\end{lemma}
\begin{proof}
For  $z=1$ or $-1$, the Jacobi-Stirling numbers  become the Legendre-Stirling numbers, and for these two special cases the lemma was proved in~\cite[Theorem~5.1]{eg}. 
It remains to prove the lemma for $-1<z<1$.

For any fixed $k\geq 1$, consider the  polynomial 
\begin{equation}\label{Bkx}
B_k(x;z)=(1-x)^{3k+2}x^{1-z}D(x^z(1-x)^{-1-3k}A_k(x;z)).
\end{equation}
By Lemma~\ref{rational:22}, the polynomial $A_k(x;z)$ is  of degree $2k$, 
it is not hard to see that  $B_k(x;z)$ is a polynomial of degree $2k+1$. 
Moreover, by~\eqref{condi}, we have $A_k(0;z)=0$, it follows from \eqref{Bkx} that $B_k(0;z)=0$.
Next we show that the nonzero roots of $A_k(x;z)$  are distinct, real and nonpositive by showing that they are intertwined with the zeros of $B_k(x;z)$. 
We proceed by induction on $k\geq 1$. For $k=1$, we have
\begin{align}
A_1(x;z)&=(1+z)x+(1-z)x^2.\label{A_1}
\end{align} 
Hence the two roots of $A_1(x;z)$ are $x_1=0$ and $x_2=\frac{z+1}{z-1}$, which is negative 
if $z\in (-1, 1)$.

Now suppose that $k\geq2$ and the zeros of $A_{k-1}(x;z)$ are  distinct nonpositive real numbers.
By Rolle's Theorem and relation~\eqref{Bkx}, the polynomial $B_{k-1}(x;z)$ has a root strictly between each pair of consecutive roots of $A_{k-1}(x;z)$; including  0, this accounts for $2k-2$ of the $2k-1$ roots of $B_{k-1}(x;z)$. To find the missing root, let $\alpha$ denote the leftmost root of $A_{k-1}(x;z)$; by \eqref{Bkx} we have $B_{k-1}(\alpha;z)=\alpha(1-\alpha)\frac{d}{dx}A_{k-1}(\alpha;z)$. Since the degree of $A_{k-1}(x;z)$ is even and its leading coefficient is positive we have 
$\lim_{x\rightarrow-\infty}A_{k-1}(x;z)=+\infty$. Now since the roots of $A_{k-1}(x;z)$ are distinct we find $\frac{d}{dx}A_{k-1}(\alpha;z)<0$; hence $B_{k-1}(\alpha;z)>0$.  But the degree of $B_{k-1}(x;z)$ is odd and his leading coefficient is positive by \eqref{Bkx}, so $\lim_{x\rightarrow-\infty}B_{k-1}(x;z)=-\infty$, and therefore $B_{k-1}(x;z)$ has a root at the left of $\alpha$. It follows that $B_{k-1}(x;z)$ has $2k-1$ distinct, real, nonpositive roots.

For example, if $k=2$ then $k-1=1$ and  we find
$$
B_1(x;z)=\left(  \left( 1+z \right)^2 x+ \left( 1-z \right)  \left( 2+z \right) 
{x}^{2} \right)  \left( 1-x \right) -4\,{x}^{2} \left( 1+z+ \left( 1-z
\right) x \right)
$$
with   $(z-1)(6+z)$ as the leading coefficient.  So $\lim_{x\to -\infty} B_1(x;z)=-\infty$. As
$B_1(x_2; z)=\frac{2(1+z)^2}{(z-1)^2}>0$, 
 there must be  a root  of $B_1(x;z)$ at the left of $x_2$.

From   \eqref{rational} and \eqref{Bkx} we deduce that  \eqref{eq:pf}  is equivalent to
\begin{equation}\label{Akx}
A_k(x;z)=x(1-x)^{3k}D((1-x)^{1-3k}B_{k-1}(x;z)).
\end{equation}
Using \eqref{Akx} and the properties of zeros of $B_{k-1}(x;z)$ we can prove similarly 
that $A_k(x;z)$ has $2k$ distinct, real, nonpositive roots.  The proof is thus complete.
\end{proof}

\begin{remark}
The constant term of $\JS(n,k;z)$ (reps. $\js(n,k;z)$) are the {\em central factorial numbers} of the second  kind  $T(2n,2k)$ (resp.  the first kind  $t(2n,2k)$)
(see~\cite[pp.~213--217]{ri} and~\cite{gz}), that is,
$$
T(2n,2k)=\JS(n,k;0), \qquad t(2n,2k)=\js(n,k;0).
$$
Since $A_k(x;0)$  can be seen as the descent polynomial of some \emph{generalized Stirling permutations} (see
 the end of~\cite{GLZ}), it follows from a result
of Brenti~\cite[Theorem 6.6.3]{Brenti89} that $A_k(x;0)$  has only real nonnegative roots.  
\end{remark}

\begin{lemma}\label{js-JS}Let $G_k(x;z)=\sum_{n\geq k}\js(n,n-k;z)x^n$. Then
$$
G_k(x;z)=(-1)^{k+1}F_k(1/x,-z).
$$
\end{lemma}
\begin{proof}
Let $g_k(n)=\js(n,n-k;z)$. Then by  recursive formulas~\eqref{eq:ja0}, for $k\geq0$, we have 
\begin{equation*}
g_k(n)=g_k(n-1)+(n-1)(n-1+z)g_{k-1}(n-1;z).
\end{equation*}
Comparing this with~\eqref{pJS} we get 
$$
g_k(n;z)=(-1)^kf_k(-n;-z).
$$
The result follows from this relation and the standard results of generating functions (cf.~\cite[Proposition~4.2.3]{st1}).
\end{proof}

\begin{proof}[{\bf Proof of Theorem~\ref{PF:diagonal}}]
By Lemma~\ref{js-JS}, we only need to prove the theorem for the sequence $\{\JS(n+k,n;z)\}_{n\geq0}$.
 When $-1\leq z\leq1$, it follows from Lemmas~\ref{PF:JSD} and~\ref{rfpo} that the sequence $\{\JS(n+k,n;z)\}_{n\geq0}$ is a PF sequence. This proves the ``if'' side of the theorem.

It remains to show the ``only if" side. When $z>1$, by~\eqref{A_1}, the polynomial $A_1(x;z)$ has a positive root $\frac{z+1}{z-1}$. Thus, by Rolle's Theorem and relationship~\eqref{Bkx}, the polynomial $B_{1}(x;z)$ has a positive  root, and so does $A_2(x;z)$ by relationship~\eqref{Akx}. It follows by induction on $k$ and the two relationships~\eqref{Bkx} and~\eqref{Akx} that $A_k(x;z)$ has a positive root for any integer $k\geq1$. The ``only if" side of the theorem then follows from Lemmas~\ref{rfpo} and~\ref{rational:22}.
\end{proof}
\begin{corollary}
The two sequences  $\{T(2(n+k),2n)\}_{n\geq0}$ and $\{t(2(n+k),2n)\}_{n\geq0}$ are PF sequences. 
\end{corollary}


\section{$z$-total positivity  of Jacobi-Stirling numbers}
\label{total:jacobi}
In this section, we show that some $z$-total positivity properties of Jacobi-Stirling numbers follow directly from the $\x$-total positivity properties of the elementary and complete homogeneous symmetric functions.
We begin with the observation that, similar to the classical Stirling numbers, the Jacobi-Stirling numbers  are also specializations of the two symmetric functions.

The {\em elementary} and {\em complete homogeneous symmetric functions} of degree $k$ in variables $x_1, x_2,\ldots, x_n$ are defined by 
\begin{align*}
&e_k(n):=e_k(x_1,x_2,\ldots,x_n)=\sum_{i_1< i_2<\ldots< i_k}x_{i_1}x_{i_2}\ldots x_{i_k},\\
&h_k(n):=h_k(x_1,x_2,\ldots,x_n)=\sum_{i_1\leq i_2\leq\ldots\leq i_k}x_{i_1}x_{i_2}\ldots x_{i_k},
\end{align*}
where $e_0(n)=k_0(n)=1$ and $e_k(n)=0$ for $k>n$. 
It is easy to deduce from the definition of $e_k(n)$ and $h_k(n)$ that
\begin{align*}
&e_k(n)=e_k(n-1)+x_ne_{k-1}(n-1),\\
&h_k(n)=h_k(n-1)+x_nh_{k-1}(n).
\end{align*}
As noticed by Mongelli~\cite{mon2},
comparing with \eqref{eq:ja0} and \eqref{eq:ja} 
one  gets immediately
the following  identities:
for $n\geq k\geq0$,
\begin{align}
&\js(n,k;z)=e_{n-k}(1(1+z),2(2+z),\ldots,(n-1)(n-1+z))\label{sjs},\\
&\JS(n,k;z)=h_{n-k}(1(1+z),2(2+z),\ldots,k(k+z))\label{sJS}.
\end{align}

The following result is due to 
Sagan~\cite[Theorem~4.4]{sa}.
\begin{lemma}
\label{pr:sy}
Let $\{x_i\}_{i\geq 1}$ be a sequence of polynomials in $q$ with nonnegative coefficients. Then, for $k\leq l$ and $m\leq n$,
$$
\text{$(i)$}\quad  e_{k-1}(n)e_{l+1}(m)\leq_q e_k(n)e_l(m);\quad
\text{$(ii)$}\quad  h_{k-1}(n)h_{l+1}(m)\leq_q h_k(n)h_l(m).
$$
Moreover, if the sequence $\{x_i\}_{i\geq 1}$ is strongly $q$-log-concave, then 
$$
\text{$(iii)$}\quad e_{k}(n+1)e_{l}(m-1)\leq_q e_k(n)e_l(m);\quad
\text{$(iv)$}\quad h_{k}(n+1)h_{l}(m-1)\leq_q h_k(n)h_l(m).
$$
\end{lemma}

\begin{proof} [{\bf Proof of Theorem~\ref{th:ja}~(i)}]
By Lemma~\ref{pr:sy} (ii) and (iv), if $x_i\in \R_+[z]$ and 
the sequence  $\{x_i\}_{i\geq 1}$ is strongly $z$-log-concave, then
\begin{align}\label{dis}
h_k(n)h_l(m)\geq_zh_{k-1}(n)h_{l+1}(m)\geq_zh_{k-1}(n+1)h_{l+1}(m-1)
\end{align}
for $k\leq l$ and $m\leq n$. As the sequence $\{i(i-1+z)\}_{i\geq1}$ is strongly $z$-log-concave, 
namely,
\begin{align*}
k(k-1+z)l(l-1+z)-(k-1)(k-2+z)(l+1)(l+z)\in \N[z]
\end{align*}
for $k\geq 1$ and $k\leq l$,
it follows from the specialization~\eqref{sJS} and \eqref{dis} that
\begin{align*}
\JS(n+k,n;z-1)\JS(m+l,m;z-1)\geq_z\JS(n+k,n+1;z-1)\JS(m+l,m-1;z-1)
\end{align*}
for $k\leq l$ and $m\leq n$, which implies the strong $z$-log-concavity of the sequence $\{\JS(n,k;z-1)\}_{k=1}^n$.
\end{proof}

\begin{remark}
Theorem~\ref{th:ja}~(i) generalizes the following log-concavity result of
Andrews et al.~\cite{aegl} and Mongelli~\cite{mo}: the sequence 
 $\{\JS(n,k;z-1)\}_{k=1}^n$ is log-concave  when  $z\geq 0$ is a  real number.
 They both proved the above result by showing  that the polynomial 
$$f_n(x)=\sum_{k\geq0}\JS(n,k;z-1)x^k$$
 has only real simple nonpositive zeros. Note that $\{a_i\}_{i=0}^d$ is a PF sequence  if and only if the polynomial $\sum_{i=0}^d a_ix^i$ has only real zeros. As the later result implies also 
that $\{\JS(n,k;z-1)\}_{k=1}^n$ is a PF sequence,  it would be interesting to see whether $\{\JS(n,k;z-1)\}_{k=1}^n$ is a
 $z$-PF sequence or not.
\end{remark}

The following theorem  was mentioned in \cite{sa} without proof. 
For convenience we include a proof.

\begin{lemma}\label{xPF}For any $n\geq1$, the two sequences
$\{h_k(n)\}_{k\geq0}$ and $\{e_k(n)\}_{k\geq0}$  are ${\bf x}$-PF sequences.
\end{lemma}
\begin{proof}
Choose any minor $M$ of the matrix $(h_{j-i}(n))_{i,j\in\N}$, say with rows $i_1,\ldots,i_d$ and columns $j_1,\ldots,j_d$. Define partitions $\lambda$ and $\mu$ by 
\begin{align*}
\lambda_k=j_d-i_k-d+k\ \ \text{and}\ \ \mu_k=j_d-j_k-d+k
\end{align*}
for $1\leq k\leq d$.
Then 
$$
M=(h_{\lambda_k-\mu_l-k+l}(n))_{k,l=1}^d.
$$
 If $\lambda_k\geq\mu_k$ for $1\leq k\leq d$, then by the Jacobi-Trudi identity~\cite[\S 7.16]{st2}, we have 
 $$\det(M)=s_{\lambda/\mu}(x_1, \ldots, x_n),
 $$
  where $s_{\lambda/\mu}$  is a polynomial in $x_1, \ldots, x_n$ with nonnegative coefficients. Otherwise, suppose $r$ is the smallest index such that $\lambda_r<\mu_r$, then
$
\det(M)=0
$
follows from the observation that $\lambda_k<\mu_l$ for all $k\geq r$ and $l\leq r$.
This completes the proof. The proof for $\{e_k(n)\}_{k\geq0}$ is similar, but using the dual Jacobi-Trudi identity~\cite[\S 7.16]{st2}.
\end{proof}

\begin{lemma}\label{le:dual}
A finite sequence $\{f_0({\bf x}),f_1({\bf x}),\ldots,f_d({\bf x})\}$ is an ${\bf x}$-PF sequence if and only if $\{f_d({\bf x}),f_{d-1}({\bf x}),\ldots,f_0({\bf x})\}$ is an ${\bf x}$-PF sequence.
\end{lemma}
\begin{proof}
By definition, a sequence $\{f_0({\bf x}),f_1({\bf x}),\ldots,f_d({\bf x})\}$ is an ${\bf x}$-PF sequence if all the minors of the matrix $(f_{j-i})_{1\leq i,j\leq n}$ are ${\bf x}$-nonnegative.
The result follows then from the fact that a matrix is  ${\bf x}$-totally positive if and only if its transpose is  ${\bf x}$-totally positive.
\end{proof}

\begin{proof} [{\bf Proof of Theorem~\ref{th:ja}~(ii)~and~(iii)}]
(ii) This follows 
immediately from the fact that $\{h_k(n)\}_{k\geq0}$ is an ${\bf x}$-PF sequence (Lemma~\ref{xPF})  and the specialization~\eqref{sJS}.

(iii) From the fact that $\{e_k(n)\}_{k\geq0}$ is an ${\bf x}$-PF sequence (Lemma~\ref{xPF})
 and~\eqref{sjs}, we see that $\{\js(n,n-k;z-1)\}_{k=0}^{n-1}$ is a $z$-P\'olya frequency sequence for any fixed $n\in\N$.
The result then follows from Lemma~\ref{le:dual}.
\end{proof}

\begin{proof}[{\bf Proof of Theorem~\ref{total:jc}}]
Fix $r,l,m\in\N$. It is well known \cite[Theorem 5.4]{sa}
that  the following matrices 
\begin{align*}
(e_{j-ri}(li+m))_{i,j\in\N}\quad\text{and}\quad
(h_{j-ri}(li+m))_{i,j\in\N}
\end{align*}
are ${\bf x}$-totally positive. The $z$-total positivity of the matrices 
$$
(\JS((l-r)i+j+m,li+m;z-1))_{i,j\geq0},\quad(\js(li+m,(r+l)i-j+m;z-1))_{i,j\geq0}
$$
 follows immediately from~\eqref{sJS} and~\eqref{sjs}. This   implies that the matrices $(\JS(n,k;z-1))_{n,k\geq0}$,  
 $(\js(n,n-k;z-1))_{n,k\geq0}$ are  $z$-totally positive.

It is known~\cite{gz} that the Jacobi-Stirling numbers are the connection coefficients of the bases
$\{x^n\}_n$ and $\{\prod_{i=0}^{n-1}(x-i(z+i))\}_n$,  namely,
\begin{gather}
x^n=\sum_{k=0}^{n}\JS(n,k;z)\prod_{i=0}^{k-1}(x-i(z+i)),\nonumber\\
\prod_{i=0}^{n-1}(x-i(z+i))=\sum_{k=0}^{n}(-1)^{n+k}\js(n,k;z)x^k.\label{JS:first}
\end{gather}
It follows that the inverse of  the matrix  $(\JS(n,k;z-1))_{n,k\geq0}$  is $(\js(n,k;z-1))_{n,k\geq0}$, up to deletion of signs. 
As the inverse of a totally positive matrix (with polynomial entries), up to deletion of signs in all entries, is also totally positive (cf.~\cite[Proposition~1.6]{pi}),  the matrix$(\js(n,k;z-1))_{n,k\geq0}$ is then $z$-totally positive.
\end{proof}

\begin{remark}
Theorem~\ref{th:ja}~(ii),~(iii) and Theorem~\ref{total:jc} are $z$-analogs  of~\cite[Theorem~5, Propositions~2 and~3]{mo}.
\end{remark}


\section{Strongly $\x$-log-convex polynomial sequences}
\label{stro:log-conv}
In this section, we investigate the log-convexity property of  the polynomials 
\begin{equation*}
\sum_{k=0}^n \JS(n,k;z)y^k,\quad \sum_{k=0}^n \js(n,k;z)y^k,\quad \sum_{k=0}^nQ_{n,k}(x,t)y^k.
\end{equation*}
We first establish a general result. 
\begin{lemma}  \label{le: 2}
For positive integers $n$ and $k$ we define 
 polynomials  $T_{n,k}$ in $\R_+[\x]$ by
\begin{equation}\label{coeffi:convex}
T_{n,k}=a_{n,k}T_{n-1,k}+b_{n,k}T_{n-1,k-1},\quad\text{for $1\leq k\leq n$},
\end{equation}
and the boundary conditions
$T_{0,0}=1$ and $T_{n,-1}=T_{n,n+1}=0$ for $n\geq1$.
\begin{itemize}
\item[(i)]
If the sequence $\{T_{n,k}\}_{k=0}^n$ is strongly ${\bf x}$-log-concave for each $n$ and 
\begin{equation}\label{ank:bnk}
a_{n,k}\geq_{\x} a_{n,k-1}\geq_{{\bf x}}0,\quad b_{n,k}\geq_{\x} b_{n,k-1}\geq_{{\bf x}}0\quad\text{for $1\leq k\leq n$},
\end{equation}
then  
\begin{equation*}
T_{m,k}T_{n,l} \geq_{{\bf x}} T_{m,l}T_{n,k}
\end{equation*}
for $0\leq m\leq n$ and $0\leq k\leq l$. 
\item[(ii)] Moreover, for fixed $j\geq0$ and $n\geq m\geq0$, if 
$
c_{j,i}\geq_{\x} c_{j,i-1}
$
for $i\geq1$,
then
$$
\sum_{i=0}^j (c_{j,j-i}-c_{j,i})T_{n,j-i}T_{m,i}\geq_{\x}0.
$$
\end{itemize}
Here $a_{n,k}, b_{n,k}$ and $c_{n,k}$ are polynomials in $\R[\x]$.
\end{lemma}

\begin{proof} 
Note that (i) implies  (ii) because 
\begin{align*}
\sum_{i=0}^j (c_{j,j-i}-c_{j,i})T_{n,j-i}T_{m,i}
=\sum_{i=0}^{\lfloor\frac{j}{2}\rfloor}(c_{j,j-i}-c_{j,i})(T_{n,j-i}T_{m,i}-T_{n,i}T_{m,j-i}).
\end{align*}
So we just need to prove  (i).

When $n=m$ or $k=l$, there is nothing to prove. So we suppose that $n>m$ and $l>k$ and proceed by induction on $n$. 
From recurrence relation \eqref{coeffi:convex}, we see that
\begin{align*} 
&T_{m,k}T_{n+1,l} - T_{m,l}T_{n+1,k}\\
=&T_{m,k}(a_{n+1,l}T_{n,l}+b_{n+1,l}T_{n,l-1})
-T_{m,l}(a_{n+1,k}T_{n,k}+b_{n+1,k}T_{n,k-1})\\
\geq&_{{\bf x}}a_{n+1,l}(T_{m,k}T_{n,l}-T_{m,l}T_{n,k})
+b_{n+1,l}(T_{m,k}T_{n,l-1}-T_{m,l}T_{n,k-1})\quad(\text{by \eqref{ank:bnk}})\\
=&a_{n+1,l}(T_{m,k}T_{n,l}-T_{m,l}T_{n,k})\\
&+b_{n+1,l}[(T_{m,k}T_{n,l-1}-T_{m,l-1}T_{n,k})+(T_{m,l-1}T_{n,k}-T_{m,l}T_{n,k-1})],
\end{align*}
which is in $\R_+[\x]$  by the induction hypothesis provided  that, for $1\leq m\leq n$ and $1\leq k\leq l$,
\begin{equation}\label{eq:1}
T_{n,k}T_{m,l}-T_{n,k-1}T_{m,l+1}\geq_{{\bf x}}0.
\end{equation}
 It remains to prove \eqref{eq:1}. We proceed  by induction on $n$. 
 As the sequence $\{T_{n,k}\}_{k=0}^n$ is strongly $\x$-log-concave, by definition,
\begin{align}\label{eq:2}
T_{n,k}T_{n,l}-T_{n,k-1}T_{n,l+1}\geq_{\x}0,
\end{align}
so, the claim is true for  $n=m$. Assume that  $n\geq m$. By recurrence relation~\eqref{coeffi:convex}, we see that
\begin{align*} 
T_{n+1,k}&T_{m,l}-T_{n+1,k-1}T_{m,l+1}\\
=&a_{n+1,k}T_{n,k}T_{m,l}+b_{n+1,k}T_{n,k-1}T_{m,l}
-a_{n+1,k-1}T_{n,k-1}T_{m,l+1}-b_{n+1,k-1}T_{n,k-2}T_{m,l+1}\\
\geq&_{{\bf x}}a_{n+1,k}(T_{n,k}T_{m,l}-T_{n,k-1}T_{m,l+1})
+b_{n+1,k}(T_{n,k-1}T_{m,l}-T_{n,k-2}T_{m,l+1}),\quad(\text{by \eqref{ank:bnk}})
\end{align*}
which is in $\R_+[\x]$  by \eqref{eq:2} and the induction hypothesis. This completes the proof of the claim \eqref{eq:1}.
\end{proof}

\subsection{Proof of Theorem~\ref{bell:JS}}
(i) 
Let $J_n(z,y)=\sum_{k=0}^n\JS(n,k;z)y^k$.
 In view of recurrence relation~\eqref{eq:ja}, we have 
\begin{align*}
&J_{m-1}(z,y)J_{n+1}(z,y)-J_{m}(z,y)J_{n}(z,y)\\
=&J_{m-1}(z,y)\sum_{k=0}^{n+1}[\JS(n,k-1;z)+k(k+z)\JS(n,k;z)]y^k\\
&-J_n(z,y)\sum_{k=0}^{m}[\JS(m-1,k-1;z)+k(k+z)\JS(m-1,k;z)]y^k\\
=&J_{m-1}\sum_{k=0}^{n}k(k+z)\JS(n,k;z)y^k-J_n(z,y)\sum_{k=0}^{m-1}k(k+z)\JS(m-1,k;z)y^k.
\end{align*}
Thus the coefficient of $y^j$ in $J_{m-1}(z,y)J_{n+1}(z,y)-J_{m}(z,y)J_{n}(z,y)$ is
\begin{equation}\label{exp:JS}
\sum_{i=0}^j [(j-i)(j-i+z)-i(i+z)]\JS(n,j-i;z)\JS(m-1,i;z).
\end{equation}
By Theorem~\ref{th:ja}~(i), the sequence $\{\JS(n,k;z)\}_{k=0}^n$ is strongly $z$-log-concave. It follows from \eqref{eq:ja} and 
Lemma~\ref{le: 2} that the expression in~\eqref{exp:JS} is nonnegative with respect to $\geq_z$ if $n\geq m\geq1$, 
which proves  (i).

(ii) 
By Eq.~\eqref{JS:first}, we have
$\sum_{k=0}^n \js(n,k;z)y^k=\prod_{i=0}^{n-1}(y+i(z+i))$.
The result can be  verified directly from this simple expression. \qed
\medskip

Recall that 
the {\em Stirling numbers of the second kind} $S(n,k)$ are defined by  the following recurrence relation 
\begin{equation*}
S(n,k)=S(n-1,k-1)+kS(n-1,k)
\end{equation*}
with $S(0,0)=1$. 
Let $B_n(y)=\sum_{k=0}^nS(n,k)y^k$ be the $n$-th Bell polynomial.
We show that Theorem~\ref{bell:JS}~(i)
 implies the following result of Chen et al.~\cite{CWY11}.  
\begin{corollary}
The Bell polynomials
$\{B_n(y)\}_{n\geq0}$ are strongly $y$-log-convex.
\end{corollary}
\begin{proof} 
By Theorem~\ref{bell:JS}~(i),  the sequence $\{J_n(z,y)\}_n$ is strongly $\{z,y\}$-log-convex, namely, the polynomial
\begin{align}\label{note}
J_{m-1}(z,y)J_{n+1}(z,y)-J_{m}(z,y)J_{n}(z,y)
\end{align}
has nonnegative coefficients.
It is known (see \cite{gz}) that $\JS(n,k;z)$ is a polynomial in $z$ of 
degree $n-k$  with leading coefficient $S(n,k)$. 
Hence  the coefficient of $z^{n-k}y^k$ in $J_n(z,y)$ is 
$S(n,k)$, which implies that 
 the coefficient of $z^{m+n-i}y^i$ in \eqref{note}
 is equal to 
that of $y^i$ in $B_{m-1}(y)B_{n+1}(y)-B_m(y)B_n(y)$ for $0\leq i\leq m+n$. 
This completes the proof of the desired result.
\end{proof}

\subsection{An open problem}
We say that a transformation of sequences $\{z_n\}_n\mapsto \{w_n\}_n$  preserves the log-convexity if the log-convexity of  $\{z_n\}_{n\geq0}$ implies that of $\{w_n\}_{n\geq0}$.
For example, 
Liu and Wang~\cite{LW10} show that the Stirling  transformation
 $w_n=\sum_{k=0}^nS(n,k)z_k$ preserves the log-convexity.
 In view of Theorem~\ref{bell:JS}, we pose the following conjecture.
 \begin{Conjecture}
 The Jacobi-Stirling transformation: $\{z_n\}_n\mapsto \{w_n\}_n$, where
 $$w_n=\sum_{k=0}^n\JS(n,k;z)z_k\quad \text{or}\quad 
 w_n=\sum_{k=0}^n\js(n,k;z)z_k,
 $$ 
 preserves the log-convexity for $z=0, 1$.
 \end{Conjecture}

\subsection{Proof of Theorem~\ref{convex}}
\label{convex:rama}
By  \eqref{def:Q}
we see that $Q_n(x,y,z,t)$ are homogeneous polynomials in $x,y,z,t$ of degree $n-1$.
As $z$ is just a homogeneous parameter, namely,
\begin{equation}\label{hom:Rama}
Q_n(x,y,z,t)=z^{n-1}Q_n(x/z,y/z,1,t/z),
\end{equation}
it suffices to study $Q_n(x,y,1,t)$. We set
\begin{equation}\label{poq}
Q_n(x,y,1,t)=\sum_{k=0}^{n-1}Q_{n,k}(x,t)y^k.
\end{equation}
Substituting \eqref{poq} in \eqref{def:Q} and  identifying the coefficients of $y^k$  we obtain $Q_{1,0}(x,t)=1$ and for $n\geq2$:
\begin{equation}\label{Qnk}
Q_{n,k}(x,t)=[x+n-1+t(n+k-1)]Q_{n-1,k}(x,t)+(n+k-2)Q_{n-1,k-1}(x,t),
\end{equation}
where $Q_{n,k}(x,t)=0$ if $k\geq n$ or $k<0$.

\begin{lemma} \label{le:1}
 For $n\geq 1$ and $l\geq k\geq1$, we have 
 $$
 Q_{n,k}(x,t)Q_{n,l}(x,t)\geq_{{\bf x}}Q_{n,k-1}(x,t)Q_{n,l+1}(x,t),
 $$
 where $\x=\{x,t\}$.
 In other words, the polynomial sequence $\{Q_{n,k}(x,t)\}_{k=0}^{n-1}$ is strongly ${\bf x}$-log-concave.
 \end{lemma}

\begin{proof}  Let 
$$
U_n(k,l)=Q_{n,k}(x,t)Q_{n,l}(x,t).
$$
 We prove by induction on $n\geq 1$. For $n=1, 2$ the inequality is trivial.
For  $n=3$,  we have
\begin{align*}
U_3(1,1) - U_3(0,2)
=6x^2+15x+10+21tx+28t+19t^2\geq_{{\bf x}}0.
\end{align*}

Using recurrence relation \eqref{Qnk} we can write 
$$
U_{n+1}(k,l)- U_{n+1}(k-1,l+1)=A_n+B_n+C_n+D_n,
$$
where 
\begin{align*} 
A_n=&\bigl(x+n+t(n+k)\bigr)\bigl(x+n+t(n+l)\bigr)\bigl[U_n(k,l)-U_n(k-1,l+1)\bigr],\\
B_n=&(n+k-1)(n+l-1)\bigl[ U(k-1,l-1)-U_n(k-2,l)\bigr],\\
C_n=&(x+n)(l-k+1)\bigl[U_n(k,l-1)-U_n(k-1,l)\bigr]\\
&+\bigl(x+n+t(n+l+1)\bigr)(n+k-2)\bigl[U_n(k,l-1)-U_n(k-2,l+1)\bigr],\\
D_n=&(l-k+1)\bigl[t^2U_n(k-1,l+1)+U_n(k-2,l)+2U_n(k,l-1)\bigr].
\end{align*}
By induction hypothesis, the polynomials  $A_n, B_n, C_n$ and $D_n$ are clearly 
nonnegative with respect to $\geq_\x$. This completes the proof.
\end{proof}
 By~\eqref{hom:Rama}, it suffices to prove Theorem~\ref{convex} for the polynomial sequence $\{Q_n(x,y,1,t)\}_{n\geq0}$.
For brevity, we write $Q_{n,k}$ for $Q_{n,k}(x,t)$, $Q_n$ for $Q_n(x,y,1,t)$ and $Q'_{n}$ for $\partial_y Q_{n}(x,y,1,t)$.  

 By recurrence relation~\eqref{def:Q}, we have
\begin{align*}
Q_{m-1}Q_{n+1}-Q_mQ_n
=(n-m+1)(y+t+1)Q_{m-1}Q_{n}+y(y+t)(Q'_{n}Q_{m-1}-Q_{n}Q'_{m-1}).
\end{align*}
Thus, the strong ${\bf x}$-log-convexity of $\{Q_n\}_{n\geq0}$ will follow from the claim that, for all $n\geq m\geq1$, 
$$
Q'_{n}Q_{m-1}-Q_{n}Q'_{m-1}\geq_{{\bf x}}0,
$$ 
where  $\x=\{x,y,t\}$. The coefficient of $y^{j-1}$ in $Q'_{n}Q_{m-1}-Q_{n}Q'_{m-1}$ is
\begin{align}\label{ram:con}
\sum_{i=0}^{j}\biggl((j-i)Q_{n,j-i}Q_{m-1,i}-iQ_{n,j-i}Q_{m-1,i}\biggr)
=\sum_{i=0}^{j}[(j-i)-i]Q_{n,j-i}Q_{m-1,i}.
\end{align}
By Lemmas~\ref{le:1}, the polynomial sequence $\{Q_{n,k}(x,t)\}_{k=0}^{n-1}$ is strongly ${\bf x}$-log-concave.
 It follows from \eqref{Qnk} and Lemma~\ref{le: 2} that the right-hand side of~\eqref{ram:con} is nonnegative with respect to $\geq_{\x}$. So the claim is true.
 This completes the proof of Theorem~\ref{convex}.
\qed

\section{On a sequence arising from Lambert $W$ function}
\label{lambert}

In  \cite{KJ10} Kalugin and Jeffrey consider another form of Lambert's equation
$we^w=x$.  Differentiating $n$ times $w$ they obtain 
\[
\frac{d^nw(x)}{dx^n}=\frac{e^{-nw(x)} p_n(w(x))}{(1+w(x))^{2n-1}},
\]
where $p_n(x)$ are polynomials that satisfy $p_1(x)=1$ and the recurrence relation
\begin{align}\label{eq:pn2}
p_{n+1}(x)=-(nx+3n-1)p_n(x)+(1+x)p_n'(x), \quad n\geq 1.
\end{align}
Based on the above recurrence, Kalugin and Jeffrey \cite{KJ10} prove  that the  coefficients of $(-1)^{n-1}p_n(x)$  are positive and 
form a unimodal sequence.
 In what follows, we show how this result follows easily from a connection with  the Ramanujan polynomials $R_n$.

\begin{proposition}\label{th:pn} We have
\begin{align*}
(-1)^{n-1}p_n(x)=(1+x)^{n-1}R_n(1/(1+x)).
\end{align*}
\end{proposition}
\begin{proof}Let 
\begin{align}\label{eq:q}
(-1)^{n-1}q_n(x)=(1+x)^{n-1}R_n(1/(1+x)).
\end{align}
and  substituting \eqref{eq:q} into \eqref{def:R} we get 
\begin{align}\label{eq:dq}
q_{n+1}(x)=-n(2+x)q_n(x)+(-1)^n(1+x)^{n-2}R_n'(1/(1+x)).
\end{align}
Now, differentiating   the   Eq.~\eqref{eq:q}  yields 
$$
(-1)^n(1+x)^{n-2}R_n'(1/(1+x))=(1+x)q_n(x)'-(n-1)q_n(x).
$$
Substituting this into \eqref{eq:dq} we see  that $q_n(x)$ satisfy recurrence \eqref{eq:pn2}. As $q_1(x)=p_1(x)$, we have  $q_n(x)=p_n(x)$ for $n\geq 1$.
\end{proof}

A sequence $a_0, a_1, \ldots, a_n$ of real numbers is said to have  no internal zeros if there do not exist  integers $0\leq i<j<k\leq n$ satisfying $a_i\neq 0$, $a_j=0$ and $a_k\neq 0$.
The following result is  known, see~\cite[Theorem 2.5.3]{Brenti89} or~\cite[Theorem 2]{Hoggar74}.
\begin{lemma}\label{th:pf} 
 If the coefficients of the polynomial $A(x)$ are nonnegative without internal zeros and log-concave, then so are the coefficients of  $A(x+1)$.
\end{lemma}
From the above proposition and lemma we can derive a neat  proof of the following result of Kalugin and Jeffrey~\cite{KJ10}. 
\begin{corollary}\label{cokj}
The coefficients of the polynomial $(-1)^{n-1}p_n(x)$ are positive, log-concave, and unimodal.
\end{corollary}
\begin{proof} 
First, by \eqref{def:R} it is clear that $R_n(y)$ is a polynomial in $y$ with positive coefficients.
By  \eqref{spec:ramanu:},  \eqref{poq}  and Lemma~\ref{le:1}, we see that the coefficients of $R_n(y)$ are  log-concave. 
 Combining these  with Proposition~\ref{th:pn} and Lemma~\ref{th:pf}, we derive that the coefficients of polynomials  $(-1)^{n-1}p_n(x)$ are positive and log-concave.
 Since a log-concave positive sequence is unimodal, we are done.
\end{proof}

\begin{remark}
Kalugin and Jeffrey~\cite{KJ10} proved Corollary~\ref{cokj} through a long discussion based on  recurrence \eqref{eq:pn2}. 
\end{remark}

\subsection*{Acknowledgments} We thank the referee for a careful reading of the manuscript. The first author was supported by the China Scholarship Council (CSC) for studying abroad. This work was also supported by CMIRA COOPERA 2012 de  la R\'egion Rh\^one-Alpes.

\end{document}